\documentclass{sl}     
\usepackage[latin1]{inputenc}
\usepackage{slsec}     
\usepackage{stud_log} 
\usepackage{slfoot}
\usepackage{slthm}     
                         
\usepackage{amsmath}
\usepackage{amssymb}

\usepackage{graphicx}

\usepackage{enumerate}

\def\kn{\kern.1em}

\renewcommand{\theequation}{\Roman{equation}}

\newtheorem{theorem}{Theorem}[section]  
\newtheorem{corollary}[theorem]{Corollary}
\newtheorem{lemma}{Lemma}[section]
\newtheorem{proposition}{Proposition}[section]

\theoremstyle{definition}
\newtheorem{definition}[theorem]{Definition}
\newtheorem{remark}{Remark}[section]
\newtheorem{example}{Example}[section]

\HeadingsInfo{F. Almi\~{n}ana, G. Pelaitay, W. Zuluaga}{On Heyting algebras with negative tense operators}

\begin{document}

\setcounter{page}{1}     

 





\threeAuthorsTitlethreelines{F.\ts Almi\~{n}ana}{G.\ts Pelaitay}
{W. \ts Zuluaga}{On Heyting algebras with \linebreak negative tense operators}

   


\begin{abstract} In this paper, we will study Heyting algebras endowed with tense negative operators, which we call tense H-algebras and  we proof that these algebras are the algebraic semantics of the Intuitionistic Propositional Logic with Galois Negations. Finally, we will develop a Priestley-style duality for H-algebras. 
\end{abstract}


\section[Introduction]{Introduction}
The language of propositional tense logics consists of the countable set {\rm Var}$:= \{x_i : i \in \omega\}$ of propositional variables, the set of logical connectives, comprising the standard connectives $\wedge$, $\vee$, $\to$, $\neg$, of which conjunction and negation can be viewed as primitive and the others as defined, the constants $\perp  := x_0 \wedge \neg x_0$ and $\top := \neg \perp$, two unary connectives $G$, $H$, whose intuitive reading is   \emph{it is always going to be the case that} and \emph{it has always been the case that}, respectively. The set of formulas {\rm For} is defined in the standard recursive way.
A tense logic is any set of formulas closed under substitution and under some consequence operation given by:

\begin{itemize}
\item[] A set of axioms, containing:
\item[(A0)] All classical tautologies,
\item [(A1)] $G(\alpha \to  \beta) \to (G\alpha \to  G\beta)$; $H(\alpha \to  \beta) \to (H\alpha \to  H\beta)$,
\item[(A2)] $\alpha \to GP\alpha$; $\alpha \to HF\alpha$, where $P\alpha:=\neg H\neg \alpha$ and $F\alpha:=\neg G\neg \alpha$.
\end{itemize}

\begin{itemize}
\item[] A set of rules, containing:

    \item [(R0)] from $\alpha$, $\alpha \to \beta$ infer $\beta$ \hfill (Detachment),
    \item [(R1)] from $\alpha$ infer $G\alpha$ \hfill (G\"{o}del rule for $G$), 
    \item [(R2)] from $\alpha$ infer $H\alpha$ \hfill (G\"{o}del rule for $H$).
\end{itemize}

The minimal tense logic {\bf K}$_t$ is the logic whose axioms are just A0--A2, and inference rules just R0--R2. It is well know that the Kripke completeness of {\bf K}$_{t}$ is provided by \emph{tense Kripke frames} $\langle X,R,R^{-1}\rangle,$ where $X$ is a set, $R$ is a binary relation on $X,$ and $R^{-1}$ is the converse of $R$ (see \cite{Burgess,Diaconescu,Kowalski}). If $x,y\in X$ and $\alpha$ is a formula, the conditions for the holding of $G\alpha$, $H\alpha$, $F\alpha$ and $P\alpha$ at a point $x$ are
\begin{equation*}
x\models G\alpha \Longleftrightarrow \forall y (xRy \Longrightarrow y\models \alpha),    
\end{equation*}
\begin{equation*}
x\models H\alpha \Longleftrightarrow \forall y (xR^{-1}y \Longrightarrow y\models \alpha),    
\end{equation*}
\begin{equation*}
x\models F\alpha \Longleftrightarrow \exists y (xRy \,\,\, \textnormal{and}\,\,\, y\models \alpha),
\end{equation*}
\begin{equation*}
x\models P\alpha \Longleftrightarrow \exists y (xR^{-1}y \,\,\, \textnormal{and}\,\,\, y\models \alpha),
\end{equation*}

We can defined the tense operators $g$, $h$, $f$ and $p$ for which we give the following conditions, (where the right--hand sides are obtained by negating the right--hand sides of the conditions above)

\begin{equation*}
x\models g\alpha \Longleftrightarrow \exists y (xRy\,\, \textnormal{and}\,\, y\not\models \alpha),    
\end{equation*}
\begin{equation*}
x\models h\alpha \Longleftrightarrow \exists y (xR^{-1}y \,\, \textnormal{and}\,\, y\not\models \alpha),    
\end{equation*}
\begin{equation*}
x\models f\alpha \Longrightarrow \forall y (xRy \Longrightarrow y\not\models \alpha),
\end{equation*}
\begin{equation*}
x\models p\alpha \Longleftrightarrow \forall y (xR^{-1}y \Longrightarrow  y\not\models \alpha),
\end{equation*}

In the system {\bf K}$_t$, the tense operators $g$, $h$, $f$ and $p$ do not introduce anything particularly new, since $g\alpha$, $h\alpha$, $f\alpha$ and $p\alpha$ are definible as $\neg G\alpha$, $\neg H\alpha$, $\neg F\alpha$ and $\neg P\alpha$ (or in equivalent way $P\neg\alpha$, $F\neg\alpha$, $H\neg\alpha$ and $G\neg\alpha$), respectively. However, with intuitionistic propositional logic the connectives $g$, $h$, $f$ and $p$ need not be definable in this way any more (see \cite{Chajda,Chajdabook,Dzik,FP12,FP14}).

The main aim of this paper is to investigate the \emph{algebraic axiomatization} of the tense operators $g$, $h$, $f$ and $p$ on Heyting algebras.  Following the terminology established in \cite{ML2022} (see also \cite{Celani,KD,Dunn,Orlowska}) for \emph{tense operators} (\emph{modal operators}, respectively) on intuitionistic logic, the tense operators $g$, $h$, $f$ and $p$ will be called \emph{negative tense operators}. 

\section{Negative tense operators on Heyting algebras}\label{s2}

In this section we will define the notion of negative tense operators on Heyting algebras. Recall that a Heyting algebra is an algebra $\mathcal{A}=\langle A, \wedge,\vee,\to,0,1\rangle$ of type $(2, 2, 2, 0, 0)$ for
which $\langle A,\wedge,\vee, 0, 1\rangle$ is a bounded distributive lattice and $\to$ is the binary operation of relative pseudocomplementation (i.e., for $a, b, c \in A$, $a \wedge c \leq b$ iff $c \leq a \to b)$. To obtain more information on this topic, we direct the reader to the bibliography indicated in \cite{Balbes,CHK}.

\begin{definition}\label{d1}
 Let $\mathcal{A}=\langle A, \wedge,\vee,\to,0,1\rangle$ be a Heyting algebra, let $g$, $h$, $f$ and $p$ be unary operations on $A$ satisfying:

\begin{itemize}
\item [(T1)] $g(1)=0$ and $h(1)=0;$
\item [(T2)] $f(0)=1$ and $p(0)=1;$
\item [(T3)] $g(x\wedge y)=g(x)\vee g(y)$ and $h(x\wedge y)=h(x)\vee h(y);$
\item [(T4)] $f(x\vee y)=f(x)\wedge f(y)$ and $p(x\vee y)=p(x)\wedge p(y);$
\item [(T5)] $gh(x)\leq x$ and $hg(x)\leq x;$
\item [(T6)] $x\leq pf(x)$ and $x\leq fp(x);$
\item [(T7)] $g(x)\wedge f(y)\leq g(x\vee y)$ and $h(x)\wedge p(y)\leq h(x\vee y);$
\item [(T8)] $f(x\wedge y)\leq f(x)\vee g(y)$ and $p(x\wedge y)\leq p(x)\vee h(y)$.
\end{itemize}
Then the algebra $(\mathcal{A}, N)$, with $N=\{ g, h, f, p\}$ will be called negative tense Heyting algebra (or tense H-algebra, for short) and $g$, $h$, $f$ and $p$ will be called negative tense operators.

\end{definition}

We will list some basic properties valid in tense H-algebras, proving just some of them.

\begin{lemma}\label{l1} Let $(\mathcal{A},N)$ be a tense H-algebra. Then
\begin{itemize}
    \item [{\rm (T9)}] $x\leq y$ implies $g(y)\leq g(x)$ and $x\leq y$ implies $h(y)\leq h(x),$
    \item [{\rm (T10)}] $x\leq y$ implies $f(y)\leq f(x)$ and $x\leq y$ implies $p(y)\leq p(x),$
    \item [{\rm (T11)}] $g(x)\leq y$ if and only if $h(y)\leq x,$
    \item [{\rm (T12)}] $x\leq p(y)$ if and only if $y\leq f(x),$
    \item [{\rm (T13)}] $h((g(x)\to y)\to y)\leq x$ and $g((h(x)\to y)\to y)\leq x$,
    \item [{\rm (T14)}] $y\leq f(x\wedge p(y))$ and $y\leq p(x\wedge f(y))$.
\end{itemize}
\end{lemma}

\begin{proof} Axioms (T9) and (T10) are a consequence of axioms
(T3) and (T4), respectively. Next, let us prove (T11). Assume that $g(x)\leq y$. From (T9), we obtain $h(y)\leq hg(x)$.  From this last statement and from (T5) it results that $h(y)\leq x$. The converse is proved in a similar way. Moreover, the proof of the validity of the axiom (T12) can be obtained in a similar way to (T11). Finally, let us verify (T13). Since $g(x)\wedge y\leq y$, we have that $g(x)\wedge (g(x)\to y)\leq y$, i.e., $g(x)\leq (g(x)\to y)\to y$. Therefore, from (T11) we obtain $h((g(x)\to y)\to y)\leq x$.  The other inequality is analogous. Finally, the proof of the validity of the axiom (T14) follows from (T12). 

\end{proof}

\begin{remark} From Lemma \ref{l1}, we have that the variety $\mathcal{HGN}$ of tense H-algebras coincide with the variety of GN-algebras (see \cite{ML2022}).
\end{remark}

\begin{example} There are a extreme example of negative tense operators  on a Heyting algebra $\mathcal{A}$. Define $g=h$ and $f=p$, such that $g(1)=0$ and $g(x)=1$ for all $x\not=1$ and $f(0)=1$ and $f(x)=0$ for all $x\not=0$.
\end{example}

\begin{example} Let $X$ be a set, $R$ be a binary relation on $X$ and $R^{-1}$ the converse of $R$. We define four operators on $\mathcal{P}(X)$ as follows:

\begin{equation}\label{g}
g_{R}(U)=\{x\in X: R(x)\cap (X\setminus U)\not=\emptyset\},    
\end{equation}
\begin{equation}\label{h}
h_{R}(U)=\{x\in X: R^{-1}(x)\cap (X\setminus U)\not=\emptyset\},    
\end{equation}
\begin{equation}\label{f}
f_{R}(U)=\{x\in X: R(x)\subseteq (X\setminus U)\},    
\end{equation}
\begin{equation}\label{p}
p_{R}(U)=\{x\in X: R^{-1}(x)\subseteq (X\setminus U)\}.   
\end{equation}

Then, we can check that the algebra $(\mathcal{P}(X),g_{R},h_{R},f_{R},p_{R})$ is a Heyting algebra with negative tense operators.

\end{example}

We will indicate an example of tense H-algebra which is not a Boolean algebra with negative tense operators.

\begin{example}\label{ej2}

Let us consider a Heyting algebra $\mathcal{A}$ visualized in Figure \ref{Fig1}.

\vspace{2cm}

\begin{figure}[h]
\begin{center}
\hspace{0.25cm}
\begin{picture}(-40,40)(0,0)
\put(00,00){\makebox(1,1){$\bullet$}}
\put(-30,30){\makebox(1,1){$\bullet$}}
\put(30,30){\makebox(1,1){$\bullet$}}
\put(00,60){\makebox(1,1){$\bullet$}}
\put(60,60){\makebox(1,1){$\bullet$}}
\put(30,90){\makebox(1,1){$\bullet$}}
\put(00,00){\line(1,1){30}}
\put(00,00){\line(-1,1){30}}
\put(-30,30){\line(1,1){30}}
\put(30,30){\line(1,1){30}}
\put(30,30){\line(-1,1){30}}
\put(00,60){\line(1,1){30}}
\put(60,60){\line(-1,1){30}}
\put(00,-10){\makebox(2,2){$0$}}
\put(-40,30){\makebox(2,2){$a$}}
\put(40,30){\makebox(2,2){$b$}}
\put(70,60){\makebox(2,2){$d$}}
\put(-10,60){\makebox(2,2){$c$}}
\put(30,100){\makebox(2,2){$1$}}
\end{picture}
\caption{Heyting algebra}
\label{Fig1}
\end{center}
\end{figure}
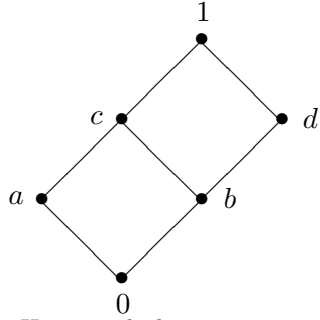

Define operators $g$, $h$, $f$ and $p$ by the table: 

\begin{center}
\begin{tabular}{|c|c|c|c|c|}\hline
   $x$   &  $g(x)$ & $h(x)$ & $f(x)$ & $p(x)$ \\ \hline
   $0$   &  $d$    & $1$ & $1$ & $1$    \\ \hline
   $a$   &  $d$    & $1$ & $a$ &  $1$  \\ \hline
   $b$   &  $d$    & $c$ & $c$ & $b$    \\ \hline
   $c$   &  $b$    & $c$ & $a$ &  $b$  \\ \hline
   $d$   &  $d$   & $0$ & $a$ &  $0$  \\ \hline
   $1$   &  $0$    & $0$ & $a$ &  $0$  \\ \hline
\end{tabular}
\end{center}

Then, it is easy to see that $(\mathcal{A}, g,h,f,p)$ is a tense H-algebra.

\end{example}

\begin{remark}  Let us consider the tense H-algebra $(\mathcal{A},g,h,f,
p)$, defined in the Example \ref{ej2}, then it can be seen: $\neg g(\neg b)=a\not=c=f(b)$ and $\neg h(\neg b)=0\not= b=p(b)$ , where $\neg x:=x\to 0$.
\end{remark}

The following lemma, will be essential for the proof of Lemma \ref{2.10}.

\begin{lemma}\textnormal{(\cite[Lemma 2.11.]{ML2022})}\label{lemma2} Let $(\mathcal{A},N)$ be a tense H-algebra and $S\subseteq A$. Then

\begin{itemize}
    \item [{\rm (a)}] If $S$ is a filter, then $f^{-1}(S)$ and $p^{-1}(S)$ are ideals.
    \item [{\rm (b)}] If $S$ is a prime filter, then $X\setminus g^{-1}(S)$ and $X\setminus h^{-1}(S)$ are filters. 
\end{itemize}
\end{lemma}

\begin{proof}

We only prove {\rm(b)}. Let $S$ be a prime filter of $A$. Then, $0\notin S$. So, by (T1), $1\in X\setminus g^{-1}(S)$. Take $a \in X\setminus g^{-1}(S)$ and $a \leq b$. Then $g(a) \notin S$ and, by (T9), $g(b) \leq g(a)$. Hence $b \notin  g^{-1}(S)$, that is $b \in  X\setminus g^{-1}(S)$. Also, let
$a \in X\setminus g^{-1}(S)$ and $b \in X\setminus g^{-1}(S)$. Then $g(a) \notin  S$ and $g(b) \notin S$, so since $S$ is prime, $g(a)\vee g(b)\notin S$ . Thus by (T3), $g(a \wedge b)\notin S$, that is $a \wedge b\in X\setminus g^{-1}(S).$ Similarlty, we can prove that $X\setminus h^{-1}(S)$ is a filter.

\end{proof}

\section{Congruences and negative tense filters}\label{s3}

Let us recall that a non-empty subset $F$ of a Heyting algebra $\mathcal{A}=\langle A,\vee,\wedge,\to,0,1\rangle$ is an implicative filter (filter for short) if: $F$ is order increasing (increasing for short) and it is closed under $\wedge$. On the other hand, let us recall that a subset $D$ of $\mathcal{A}$ is a deductive system if it satisfies:

\begin{itemize}
    \item [(D1)] $1\in D,$
    \item [(D2)] $x$, $x\to y\in D$ implies $y\in D$.
\end{itemize}
Let $F$ be a filter of $\mathcal{A}$ and consider the set 
\[\theta_{F} = \{(x, y) \in A \times A : x \leftrightarrow y \in F\},\] 
where as usual, $x \leftrightarrow y$ must be taken as $(x\rightarrow y)\wedge (y\rightarrow x)$. In \cite{Monteiro}, Monteiro proved that in every Heyting algebra, the notions of
filter and deductive system both coincide (see also \cite{CHK}). Moreover, if we write ${\rm Con}(A)$ for the lattice of all congruences on $\mathcal{A}$, Monteiro (\emph{op.cit}) showed that the assignments  $F\mapsto \theta_{F}$ and $\theta \mapsto F_{\theta}=[1]_{\theta}$ determine a poset isomorphism between ${\rm Con}(A)$ and the set of filters of $\mathcal{A}$, where $[x]_\theta$ stands for the equivalence class of $x$ modulo $\theta$.
\\

In order to characterize the lattice {\rm Con}$_{t}$(A) of all congruences on a tense H-algebra $(\mathcal{A},N)$, we introduce the following definitions:

\begin{definition}
Let $(\mathcal{A},N)$ be a tense H-algebra. A congruence on $(\mathcal{A},N)$ is a Heyting congruence $\theta$ which is compatible with every $u\in N$. I.e. if $(x,y)\in \theta$, then $(u(x),u(y))\in \theta$, for every $u\in N$.
\end{definition}

\begin{definition}\label{def tnfilter}
Let $(\mathcal{A},N)$ be a tense H-algebra. A filter $F$ of $\mathcal{A}$ is said to be a negative tense filter provided that:
\[x\rightarrow y \in F\; \text{implies}\; u(y)\rightarrow u(x)\in F, \]
for every $u\in N$.
\end{definition}

\begin{lemma}\label{lem equivalent tnfilter}
Let $(\mathcal{A},N)$ be a tense H-algebra and let $F\subseteq A$. Then, the following are equivalent:
\begin{enumerate}
\item $F$ is a negative tense filter.
\item $x\leftrightarrow y \in F$ implies $u(x)\leftrightarrow u(y)\in F$, for every $u\in N$.
\item The following hold:
\begin{enumerate}
    \item  $g(x)\to y\in F$ implies $h(y)\to x\in F,$
    \item  $h(x)\to y\in F$ implies $g(y)\to x\in F,$
    \item  $x\to p(y)\in F$ implies $y\to f(x)\in F,$
    \item  $x\to f(y)\in F$ implies $y\to p(x)\in F.$
\end{enumerate}
\end{enumerate}
\end{lemma}
\begin{proof}
$(1)\Rightarrow (2)$. Let us assume $x\leftrightarrow y \in F$. Since $F$ is increasing, $x\rightarrow y, y\rightarrow x \in F$. So $u(y)\rightarrow u(x), u(x)\rightarrow u(y)\in F$, for every $u\in N$. Therefore, since $F$ is closed by $\wedge$, $u(x)\leftrightarrow u(y)\in F$, as required.

\noindent $(2)\Rightarrow (1)$. Let us assume $x\rightarrow y \in F$. Then $(x\vee y)\leftrightarrow y\in F$. From  (2), $u(x\vee y)\leftrightarrow u(y)\in F$, for every $u\in N$. Observe that since $u(x\vee y)\leq u(y)$, for every $u\in N$, the latter implies that $u(y)\rightarrow u(x\vee y)\in F$. Therefore, due to $u(x\vee y)\leq u(x)$, we get $u(y)\rightarrow u(x\vee y)\leq u(y)\rightarrow u(x)$. So, because $F$ is increasing, $u(y)\rightarrow u(x)\in F$ for every $u\in F$, as claimed.

\noindent $(1)\Rightarrow (3)$. We only prove (a) and (c), because the proofs for (b) and (d) are analogue. For (a), let us assume $g(x)\rightarrow y\in F$. Then, by (1), $h(y)\rightarrow h(g(x))\in F$. Since $h(g(x))\leq x$,  it is the case that $h(y)\rightarrow h(g(x))\leq h(y)\rightarrow x$. So, due to $F$ is increasing, $h(y)\rightarrow x\in F$. For (c), if $x\to p(y)\in F$, by (1), $f(p(y))\rightarrow f(x)\in F$. Since $y\leq f(p(y))$, then we get $f(p(y))\rightarrow f(x)\leq y\rightarrow f(x)$. Hence, because $F$ is increasing, we conclude $ y\rightarrow f(x)\in F$.

\noindent $(3)\Rightarrow (1)$. Suppose $x\rightarrow y \in F$. We only prove $g(y)\rightarrow g(x)\in F$ and $f(y)\rightarrow f(x)\in F$, due to the proofs for the remaining cases are similar. On the one hand, since $h(g(x))\leq x$ we obtain that $x\to y\leq h(g(x))\to y$. Hence, from the hypothesis, we get $h(g(x))\to y\in F$.  From this last assertion and (b), we infer that $g(y)\to g(x)\in F$.  On the other hand, from $y\leq p(f(y))$ we deduce that $x\to y\leq x\to p(f(y))$. Then, $x\to p(f(y))\in F$. So, from (c), we can conclude $f(y)\to f(x)\in F$, as desired.
\end{proof}

Let $(\mathcal{A},N)$ be a tense H-algebra. In what follows we will denote by $\mathcal{F}_{t}(A)$ the set of all negative tense filters of $(\mathcal{A},N)$.

\begin{lemma}\label{l3} 
Let $(\mathcal{A},N)$ be a tense H-algebra. For any $F\in \mathcal{F}_{t}(A)$, the Heyting algebra congruence $\theta_{F}$ is a congruence on $(\mathcal{A},N)$. 
\end{lemma}

\begin{proof} Let $F \in \mathcal{F}_{t} (A)$, since $\mathcal{A}$ is a Heyting algebra
and $F$ is a filter of $\mathcal{A}$, then we know that $\theta_{F}\in {\rm Con}(A)$. Let us prove that $\theta_{F}$ preserves the negative tense operators.  Let $(x, y) \in \theta_{F}$. Then, from Lemma \ref{lem equivalent tnfilter} (2), we have $g(x)\leftrightarrow g(y) \in F$, i.e., $(g(x),g(y))\in \theta_{F}$. In a similar fashion, it can be proved that $\theta_{F}$ preserves $h$, $f$ and $p$. 
\end{proof}

\begin{lemma}\label{l4} 
Let $(\mathcal{A},N)$ be a tense H-algebra. Then for any $\theta\in {\rm Con}_{t}(A),$  $[1]_{\theta}$ is a negative tense filter of $(\mathcal{A},N)$. 
\end{lemma}

\begin{proof} Let $\theta \in {\rm Con}_{t} (A)$. Then $\theta \in Con(A)$ and consequently, we have that $[1]_\theta$ is a deductive system of $A$. In order to prove our claim, we will use Lemma \ref{lem equivalent tnfilter} (3). We start by assuming that $g(x)\to y\in [1]_\theta$. I.e. $(g(x)\to y,1)\in \theta$. Since $\theta$ is compatible with $\to$ and $h$ we can infer that $(h((g(x)\to y)\to y)\to x, h(y)\to x)\in \theta$.  From this last statement and (T13) we obtain $(1,h(y)\to x)\in \theta$, i.e, $h(y)\to x\in [1]_\theta$. On the other hand, if $x\to p(y)\in [1]_\theta$ thus $(x\to p(y),1)\in \theta$. Since $\theta$ is compatible with $\wedge$, we have that $(x\wedge (x\to p(y)),x)\in \theta$, i.e., $(x\wedge p(y),x)\in \theta$. Taking into account
that $\theta$ preserve $f$ and $\to$  we infer that $(y\to f(x\wedge p(y)),y\to f(x))\in \theta$. Then, from (T14), we have $(1,y\to f(x))\in \theta$. Therefore, $y\to f(x)\in [1]_\theta$. The proofs of the remaining cases are similar.
\end{proof}

The following result immediately follows from Lemmas \ref{l3} and \ref{l4}:

\begin{theorem}\label{t1} 
Let $(\mathcal{A},N)$ be a tense H-algebra. Then, the assignments $F \mapsto \theta_{F}$ and $\theta \mapsto F_{\theta}$ extend to an isomorphism between the lattices $\mathcal{F}_{t}(\mathcal{A})$ and ${\rm Con}_{t}(\mathcal{A})$.

\end{theorem}

\subsection{Negative tense filter generation}\label{Negative tense filter generation}

Let $\mathcal{A}$ be a Heyting algebra and let $u$ be a unary operator on $A$. We recall from \cite{H2001} that from every $a\in A$, we can consider:

\[[u](a)=\bigwedge \{u(b)\leftrightarrow u(c)\colon a\leq b\leftrightarrow c\}. \]

It is clear from the definition that $[u]$ does not have to exist in general. So it may be considered as a partial function on $A$. Nevertheless, if $u$ is antitone, then $[u]$ exists and it has a particular form:

\begin{lemma}\label{lem [u]}
Let $\mathcal{A}$ be a Heyting algebra and let $u$ be an antitone operator on $A$. Then, 
\[[u](a)=\bigwedge \{u(a\wedge b)\rightarrow u(b)\colon b\in A\}. \]
\end{lemma}
\begin{proof}
On the one hand, since $a\leq b\rightarrow a$. Then, it is the case that $a\leq b\leftrightarrow (a\wedge b)$, for every $b\in A$. Hence,  since $u(b)\leq u(a\wedge b)$, we get:
\[[u](a)\leq u(b)\leftrightarrow u(a\wedge b)= u(a\wedge b) \rightarrow u(b), \]
for every $b\in A$. Therefore, $[u](a)\leq \bigwedge \{u(a\wedge b)\rightarrow u(b)\colon b\in A\}$. On the other hand, if $a\leq b\leftrightarrow c$, then it follows that $a\wedge b= a\wedge c$. Observe that the latter, together with a well known property of Heyting algebras, implies that:
\[(u(b)\leftrightarrow u(a\wedge b)) \wedge (u(a\wedge b)\leftrightarrow  u(a\wedge c)) \wedge (u(a\wedge c)\leftrightarrow u(c))\leq u(b)\leftrightarrow u(c),\]
which means that, for every $b,c\in A$:
\[(u(a\wedge b)\rightarrow u(b)) \wedge (u(a\wedge c)\rightarrow u(c))\leq u(b)\leftrightarrow u(c). \]
Hence, $\bigwedge \{u(a\wedge b)\rightarrow u(b)\colon b\in A\}\leq [u](a)$. This concludes the proof.

\end{proof}

Now we specialize Lemma \ref{lem [u]} for tense H-algebras.

\begin{lemma}\label{lem [u] tense H-algebras}
Let $(\mathcal{A},N)$ be a tense H-algebra. Then, $[g](a)=\neg g(a)$ and $[h](a)=\neg h(a)$.
\end{lemma}
\begin{proof}
We only proof the statement for $g$ due to the proof for $h$ is analogue. On the one hand, notice that since $g(a\wedge 1)\rightarrow g(1)=\neg g(a)$, it is clear that $[g](a)\leq \neg g(a)$. On the other hand, since for every $b\in A$, we have $g(b)\wedge \neg g(a)\leq g(b)$, then $(g(a)\wedge \neg g(a))\vee (g(b)\wedge \neg g(a))\leq g(b)$. Thus, $(g(a)\vee g(b))\wedge \neg g(a)\leq g(b)$. So, by (T3) and residuation, we get $\neg g(a)\leq g(a\wedge b)\rightarrow g(b)$. Therefore, $\neg g(a)\leq [g](a)$ and consequently, $[g](a)=\neg g(a)$, as desired. 
\end{proof}

Let $(\mathcal{A},N)$ be a tense H-algebra. Inspired in \cite{H2001}, we consider the following operator:
\[[N](a):=[f](a)\wedge [g](a) \wedge [h](a) \wedge [p](a). \]
It is no hard to see that $[N]$ is defined for every $a\in A$, as long as $[u]$ is defined, for every $u\in N$. The following result, shows that in some sense, the relevant information on $[N]$ underlies on the information provided by the operators $g$ and $h$.

\begin{lemma}\label{lem [N] tense H-algebras}
Let $(\mathcal{A},N)$ be a tense H-algebra. Then, $[N](a)=\neg g(a) \wedge \neg h(a)$.
\end{lemma}
\begin{proof}
It is clear from Lemma \ref{lem [u] tense H-algebras} and the definition of $[N]$, that $[N](a)\leq \neg g(a) \wedge \neg h(a)$. So in order to proof the remaining inequality, observe that from (T8), it is the case that $\neg g(a) \wedge f(a\wedge b)\leq \neg g(a) \wedge (f(b)\vee g(a))=\neg g(a)\wedge f(b)\leq f(b)$, for every $b\in A$. Then, by residuation $\neg g(a)\leq f(a\wedge b)\rightarrow f(b)$, for every $b\in A$. Hence, $\neg g(a)\leq [f](a)$. In a similar fashion it can be proved that $\neg h(a)\leq [p](a)$. Therefore, from the latter we can conclude $\neg g(a) \wedge \neg h(a) \leq [N](a)$. This concludes the proof.
\end{proof}

\begin{remark}\label{rem [u] are normal}
We stress that by Theorem 2.3 of \cite{H2001}, in every tense H-algebra $(\mathcal{A},N)$, for every $u\in N$, the operator $[u]$ is a normal operator. I.e. $[u](1)=1$ and $[u](a\wedge b)=[u](a)\wedge [u](b)$. It is no hard to see that the latter implies that the operator $[N]$ is also normal. 
\end{remark}

Let $(\mathcal{A},N)$ be a tense H-algebra and let $F$ be a filter of $\mathcal{A}$. We say that $F$ is a \emph{$[N]$-filter}, if it is closed by $[N]$. I.e. $[N](x)\in F$ whenever $x\in F$. For a subset $X \subseteq A$, we denote by $[X)_N$ the smallest negative tense filter containing $X$ and $[X)_{[N]}$ the smallest $[N]$-filter containing $X$. In particular, for every $a\in A$ we write $[a)_N$ and $[a)_{[N]}$ instead of $[\{a\})_N$ and $[\{a\})_{[N]}$, respectively.

\begin{lemma}\label{lem negative tense filters and [N]-filters}
In every tense H-algebra $(\mathcal{A},N)$, negative tense filters and $[N]$-filters coincide.
\end{lemma}
\begin{proof}
We start by noticing that from Lemma 3.1 (1) of \cite{H2001} every $[N]$-filter is a negative tense filter. For the converse, let $F$ be a negative tense filter of $\mathcal{A}$ and let us assume $a\in F$. Since $a=1\rightarrow a$ then, from the hypothesis and the definition of negative tense filter, we get that $\neg g(a)=g(a)\rightarrow g(1)\in F$. Similarly we obtain that $\neg h(a)\in F$. Therefore, from the assumption on $F$ and Lemma \ref{lem [N] tense H-algebras}, we get $[N](a)\in F$, as desired.
\end{proof}

For $a \in A$ we define $[N]^n$ inductively on $n$ as follows: 
\begin{displaymath}
\begin{array}{cccc}
[N]^0(a)=a, & & & [N]^{n+1}(a)=[N]([N]^{n}(a)).
\end{array}
\end{displaymath}

Furthermore, for every $k\in \mathbb{N}$, we set $[N]^{(k)}(a):=[N]^{(0)}(a)\wedge \ldots \wedge [N]^{(k)}(a)$.

\begin{theorem}\label{negative tense generated filter}
Let $(\mathcal{A},N)$ be a tense H-algebra and let $X\subseteq A$. Then,
\[ [X)_N=\{ y\in A\colon\!\! [N]^{(k)}(x_1\wedge \ldots \wedge x_n)\leq y, \text{for some}\; x_1,\ldots, x_n\in X \text{and}\, k\in \mathbb{N}\}. \]
\end{theorem}
\begin{proof}
It follows from Lemma \ref{lem negative tense filters and [N]-filters} and Lemma 3.1 (2) and Theorem 3.2 of \cite{H2001}.
\end{proof}

We conclude this part with an immediate consequence of Theorem \ref{negative tense generated filter}, the definition of simple algebra and Corollary 3.3 of \cite{H2001}. The following result provides a characterization of simple and subdirectly irreducible tense H-algebras.

\begin{corollary}\label{subdirectly irreducible and simple tH-algebras}
Let $(\mathcal{A},N)$ be a tense H-algebra. Then, the following hold:
\begin{enumerate}
\item $(\mathcal{A},N)$ is subdirectly irreducible if and only if for every $a\in A-\{1\}$, there exist some $b\in A$ and $k\in \mathbb{N}$ such that $[N]^{(k)}(a)\leq b$.
\item $(\mathcal{A},N)$ is simple if and only if for every $a\in A-\{1\}$, there exists some $k\in \mathbb{N}$ such that $[N]^{(k)}(a)= 0$.
\end{enumerate}
\end{corollary}

\subsection{A deduction-detachment theorem for $\textbf{IGN}$}

The Intuitionistic Propositional Logic with Galois Negations (\textbf{IGN}) was introduced in \cite{ML2022} as an expansion of intuitionistic propositional logic (\textbf{IPL}) by two pairs of Galois negations $f,p,g,h$. Such a logic can be considered as the order dualization of Ewald's intuitionistic tense logic $\textbf{IK}_t$ in the sense that Galois pairs of negations are order duals of residuated pairs of tense operators. Such a logic may be defined by the following Hilbert-style calculus:

\vspace{0.2cm}
\textbf{Hilbert axioms}:

\begin{itemize}
\item[(1)] All instances of \textbf{IPL}.
\item[(2)] \(  g\varphi \wedge f\psi \to g(\varphi \vee \psi) \)
\item[(3)] \( f(\varphi \wedge \psi) \to f\varphi \vee g\psi  \)
\item[(4)] \( h\varphi \wedge p\psi \to h(\varphi \vee \psi) \)
\item[(5)] \( p(\varphi \wedge \psi) \to p\varphi \vee h\psi \)
\end{itemize}

\vspace{0.2cm}

\textbf{Hilbert rules}:

\begin{itemize}
    \item[(MP)] \( \{ \varphi \to \psi, \varphi \} \rhd \psi \)
    \item[(RN1)] \( \varphi \rightarrow f\psi \rhd \psi \rightarrow p\varphi\) 
    \item[(RN2)] \(g\varphi \rightarrow \psi \rhd h\psi \rightarrow \varphi. \)
\end{itemize}

Our aim now is to discuss how the logic \( \mathbf{IGN} \) can be algebrized by the variety $\mathcal{HGN}$ in a similar fashion as \( \mathbf{IPL} \) can be algebraized by the variety of Heyting algebras. To do so, we need to recall some notions fist. Let $\mathbf{L}$ be a logic \footnote{A substitution invariant consequence relation (see \cite{BP}).} with a propositional language $\mathcal{L}$.  We write $Fm_{\mathcal{L}}$ for the set of $\mathcal{L}$-formulas formed in the usual way, over a countable set of propositional variables and $\mathbf{Fm}_{\mathcal{L}}$ for the absolutely free algebra of type $\mathcal{L}$. An \emph{${\mathcal{L}}$-equation} is an ordered pair in $(\varphi,\psi)\in Fm_{\mathcal{L}}$. We usually denote the equation $(\varphi,\psi)$ by $\varphi\approx\psi$. The set of all ${\mathcal{L}}$-equations is denoted by $Eq_{{\mathcal{L}}}$. Recall that if $\mathcal{K}$ is a class of similar algebras and $\Theta\cup\{\epsilon \approx \delta\}$ is a set of equations in the type of $\mathcal{K}$, then $\Theta\vDash_{\mathcal{K}}\epsilon \approx \delta$ means that for every $\mathbf{A}\in \mathcal{K}$ and every assignment $m$ of variables into $\mathbf{A}$, if $m(\alpha)=m(\beta)$ for every $\alpha\approx\beta\in\Theta$, then $m(\epsilon)=m(\delta)$. We say that a logic  $\mathbf{L}$ is \emph{Block-Pigozzi algebraizable} (algebraizable, for short) \cite{BP} if there is a class of algebras \( \mathcal{K} \) on the signature of \( \mathbf{Fm} \) and structural transformers \( \tau \colon \mathrm{Fm}_{\mathcal{L}} \to \mathcal{P}\left(\mathrm{Eq}_{\mathcal{L}} \right) \), \( \rho \colon \mathrm{Eq}_{\mathcal{L}} \to \mathcal{P}\left(\mathrm{Fm}_{\mathcal{L}}\right) \) such that for each \( \Gamma, \varphi \subset \mathrm{Fm}_{\mathcal{L}} \) and \( \Theta, \alpha \approx \beta \in \mathrm{Eq}_{\mathcal{L}} \) the following conditions hold:
\begin{itemize}
\item[(1)] \( \Gamma \vdash_{\mathbf{L}} \varphi \) if and only if \( \tau[\Gamma] \vDash_{\mathcal{K}} \tau(\varphi) \),
\item[(2)] \( \varphi \approx \psi \models_{\mathcal{K}} \tau\left( \rho(\varphi \approx \psi) \right) \).
\end{itemize}

It is well known that the logic \( \mathbf{IPL} \) is algebraizable by the variety of Heyting algebras with structural transformers $\tau$ and $\rho$ defined by \( \tau(\varphi)= \{ \varphi \approx 1 \} \) and \( \rho(\varphi \approx \psi) := \{ \varphi \to \psi, \psi \to \varphi \} \), respectively. We stress that with these transformers we can make the logic $\mathbf{IGN}$ algebraizable by the variety $\mathcal{HGN}$. To this end, let \( \Gamma \) be a theory of \( \mathbf{IGN} \). We define the relation $\Omega(\Gamma)$ by
\[ \Omega(\Gamma)= \{ (\varphi,\psi) \in Fm_{\mathcal{L}} \times Fm_{\mathcal{L}} \colon  \varphi \to \psi, \psi \to \varphi \in \Gamma \}. \] 
Straightforward computations show that \( \Omega(\Gamma) \) is a congruence on \( \mathbf{Fm}_{\mathcal{L}} \) compatible with \( \Gamma \), in the sense that \( \gamma \in \Gamma \) if and only if \( \left(\gamma,\top \right) \in \Omega\left( \Gamma \right) \). We write \( [\varphi]_{\Gamma} \) for the equivalence class of \( \varphi \). Let us consider the algebra 
\[ \textbf{Fm}_{\mathcal{L}}/\Omega(\Gamma)= \langle Fm_{\mathcal{L}}/\Omega(\Gamma), \vee, \wedge, \to , f,g,h,p ,\perp , \top \rangle \]
of type \( \mathcal{L}=\{\vee, \wedge, \to , f,g,h,p,\top,\perp\} \) with $\ast:=[ \ast]_{\Gamma}$, for all $\ast \in \mathcal{L}$. The details of the proofs of the following two results are left to the reader.

\begin{lemma} \label{algebra LT} The algebra \( \textbf{Fm}_{\mathcal{L}}/\Omega(\Gamma) \) is a tense H-algebra. Moreover, the relation \( [\varphi]_{\Gamma}  \leq_{\Gamma} [\psi]_{\Gamma}  \) if and only if \( \varphi \to \psi \in \Gamma \) is a partial order on \( \mathbf{Fm}_{\mathcal{L}}/\Omega(\Gamma) \), with largest element \( [\top]_{\Gamma} \).
\end{lemma}

\begin{theorem} \label{algebraizacion} The logic \(  \mathbf{IGN} \) is algebraizable with equivalent algebraic semantic given by the variety \( \mathcal{HGN} \) and the structural transformers \( \tau(\varphi)=\{ \varphi \approx 1 \} \) and  \( \rho(\varphi \approx \psi) := \{ \varphi \to \psi, \psi \to \varphi \} \).

\end{theorem}

By last, we prove that the logic $ \mathbf{IGN}$ enjoys of a meta-logical property called the local deduction-detachment theorem, which is a generalization of the classical deduction-detachment theorem. To do so, we will make use of the results we obtained in Section \ref{Negative tense filter generation}. We recall that a logic $\mathbf{L}$ has the \textit{local deduction-detachment theorem} (or \emph{LDDT}) if there exists a family $\{d_j(p,q)\colon j\in J\}$ of sets $d_j(p,q)$ of formulas in at most two variables such that for every set $\Gamma\cup \{\varphi,\psi\}$ of formulas in the language of $\mathbf{L}$:
\begin{displaymath}
\begin{array}{ccc}
\Gamma, \varphi \vdash_{\mathbf{L}} \psi &  \Longleftrightarrow & \Gamma  \vdash_{\mathbf{L}} d_{j}(\varphi,\psi)\; \text{for some}\; j\in J.
\end{array}
\end{displaymath}

If $\mathcal{V}$ is a variety and $X$ is a set, we denote by $\mathbf{F_{\mathcal{V}}}(X)$ the $\mathcal{V}$-free algebra over $X$. Moreover, if $\varphi$ is a formula in the language of $\mathbf{L}$, then we write $\bar{\varphi}$ for the image of $\varphi$ under the natural map ${\bf Fm}(X)\to \mathbf{F}_{\mathcal{V}}(X)$ from the term algebra $\mathbf{Fm}(X)$ over $X$ onto $\mathbf{F}_{\mathcal{V}}(X)$. If $\Gamma$ is a set of formulas, we also denote by $\bar{\Gamma}$ the set $\{\bar{\varphi} : \varphi\in\Gamma\}$. The following is a technical result which is essentially restatement of Lemma 2 of \cite{MMT2014}.

\begin{lemma}\label{lem: Technical lemma}
Let $\Theta\cup \{\varphi\approx\psi\}$ be a set of equations in the language of $\mathcal{V}$, and let $X$ be the set of variables occurring in $\Theta\cup\{\varphi\approx\psi\}$. Then the following are equivalent:
\begin{enumerate}
\item $\Theta \models_{\mathcal{V}} \varphi \approx \psi$.
\item $(\bar{\varphi}, \bar{\psi})\in \bigvee_{\epsilon\approx\delta\in\Theta} \mathsf{Cg}^{\mathbf{F}_{\mathcal{V}}(X)}(\bar{\epsilon},\bar{\delta})$.
\end{enumerate}
\end{lemma}

\begin{theorem}\label{theo: HGN satisfies parametrized LDDT}
Suppose that $\mathbf{L}$ is an axiomatic extension of $\textbf{IGN}$ that is algebraized by the subvariety $\mathcal{V}$ of $\mathcal{HGN}$ and let $\mathcal{L}$ be the language of $\mathbf{L}$. Further, let $\Gamma\cup \Delta \cup \{\psi\} \subseteq Fm_{\mathcal{L}}$. Then $\Gamma, \Delta \vdash_{\mathbf{L}} \psi$ if and only if for some $n\geq 0$ there exist $k\in \mathbb{N}$ and $\psi_{1},\ldots,\psi_{m}\in \Delta$ such that $\Gamma \vdash_{\mathbf{L}} [N]^{(k)}(\psi_{1}\wedge \ldots \wedge \psi_{m})\rightarrow \psi$.
\end{theorem}
\begin{proof}
We only present the left-right direction of the proof due to the proof of the converse is analogue. Let us assume that \( \Gamma \cup \Delta \vdash_{\mathbf{IGN}} \varphi \). Then from Theorem \ref{algebraizacion} we get that $ \{ \psi \approx 1 \colon \psi \in \Gamma \cup \Delta \} \vDash_{\mathcal{HGN}} \varphi \approx 1 $. Therefore, by Lemma \ref{lem: Technical lemma} it is the case that \( \left( \overline{\varphi},\overline{1} \right) \in \bigvee_{\psi \in \Gamma \cup \Delta} \mathsf{Cg}^{\mathbf{F}_{\mathrm{V}}(X)} \left( \overline{\psi},\overline{1} \right) \).  Notice that as a consequence of Theorem \ref{t1} and the latter, we get $\overline{\varphi} \in [\overline{\Gamma} \cup \overline{\Delta})_{N}$. Hence, by Theorem \ref{negative tense generated filter}, there exist $k\in \mathbb{N}$ and $\overline{\psi}_{1},\ldots,\overline{\psi}_{n}\in \overline{\Gamma} \cup \overline{\Delta}$ such that 

\[[N]^{(k)}(\overline{\psi}_{1}\wedge \ldots \wedge\overline{\psi}_{n})\leq \overline{\varphi}.\]

Let \( C = \{ i \in \{1,...,n\} \colon \psi_{i} \in \Gamma \} \) and \( D = \{1,..,n\} \setminus C \), then from Remark \ref{rem [u] are normal} we have that 
\[ \bigwedge_{i \in C} [N]^{(k)}\left( \overline{\psi}_{i} \right) \wedge  \bigwedge_{j \in D} [N]^{(k)}\left( \overline{\psi}_{j} \right) \leq \overline{\varphi}. \]
Thus, by residuation we have that 
\[ \bigwedge_{i \in C} [N]^{(k)}\left( \overline{\psi}_{i} \right) \leq  \bigwedge_{j \in D} [N]^{(k)}\left( \overline{\psi}_{j} \right) \to \overline{\varphi}. \]
From Theorem \ref{negative tense generated filter}, the latter means that \( \bigwedge_{j \in D} [N]^{(k)}\left( \overline{\psi}_{i} \right) \to \overline{\varphi} \in  [\overline{\Gamma})_{N} \). Finally, by Lemma \ref{lem: Technical lemma}, Theorem \ref{algebraizacion} and Remark \ref{rem [u] are normal}, we have that \(\Gamma \vdash_{\mathbf{L}} [N]^{(k)}(\bigwedge_{j \in D} \psi_{j})\rightarrow \varphi \), as claimed.
\end{proof}

We stress that if we take $\Delta = \{\varphi\}$ and we consider $d_{l} (p, q) = [N]^{l}(p) \rightarrow q$ for $l \in \mathbb{N}$, then from Theorem \ref{theo: HGN satisfies parametrized LDDT}, we can conclude the following:

\begin{corollary}\label{LDDT}
The logic $\textbf{IGN}$ has the LDDT.
\end{corollary}

\section{Duality for negative tense Heyting algebras}\label{Duality for negative tense Heyting algebras}

We recall that Esakia duality for Heyting algebras (see \cite{Esakia}) establishes a dual equivalence between the category $\mathbb{HA}$
of Heyting algebras and homomorphisms of Heyting algebras and the category $\mathbb{HS}$ of Heyting spaces and $p$-continuous morphisms (called Heyting morphisms),

\begin{center}
$\mathfrak{X}: \mathbb{HA}\leftrightarrows \mathbb{HS}^{\rm op}:{D}$
\end{center}

For every Heyting algebra $A,$ $\mathfrak{X}(A)$ denotes the set of prime filters of $A$. For every Heyting space $X,$ $D(X)$ denotes the set of clopen upsets of $X$.

We have that $\sigma_{A}(x):=\{P\in \mathfrak{X}(A): x\in P\}$ is an isomorphism of Heyting algebras between $A$ and $D(\mathfrak{X}(A))$ and $\varepsilon_{X}(x):=\{U\in D(X): x\in U\}$ is an isomorphism of Heyting spaces between $X$ and $\mathfrak{X}(D(X))$. Both isomorphisms are natural.

In this section, we extend Heyting duality to the category of tense H-algebras. In what follows, if $(X,\leq)$ is a poset (i.e. partially ordered set) and $Y \subseteq X$, then we will denote by $\downarrow Y$ $(\uparrow  Y )$ the set of all $x \in X$ such that $x \leq y$ $(y \leq x)$ for some $y \in Y$. 

\begin{definition} A tense H-space is a system $(X,R)$ where $X$ is a Heyting space, $R$ is a binary relation on $X$ and $R^{-1}$ is the converse of $R$ such that the following conditions are satisfied:

\begin{itemize}
    \item [(S1)] For each $x \in X$, $R(x)$ and $R^{-1}(x)$ are closed subsets of $X$.
    \item [(S2)] For each $x\in X$, $R(x)=\downarrow R(x)\cap \uparrow R(x)$.
    \item [(S3)] For each $U\in D(X)$ it holds that $g_{R}(U),$ $h_{R}(U)$, $f_{R}(U)$, $p_{R}(U)\in D(X),$  where $g_{R}(U), h_{R}(U), f_{R}(U)$ and $p_{R}(U)$ are defined as in \ref{g}, \ref{h}, \ref{f} and \ref{p}, respectively.
\end{itemize}
\end{definition}

\begin{definition} A tense H-function from a tense H-space $(X_1,R_1)$ into another one $(X_2,R_2)$, is a Heyting morphism $k:X_{1}\longrightarrow X_{2}$, which satisfies the following conditions: 
\begin{itemize}
    \item[(m1)] $(x,y)\in R_1$ implies $(k(x),k(y))\in R_2$,
    \item [(m2)] if $(k(x),y)\in R_2$, then there is   $z\in X_1$ such that $(x,z)\in R_1$ and $k(z)\leq y$, 
    \item[(m3)] if $(y,k(x))\in R_2$, then there is   $z\in X_1$ such that $(z,x)\in R_1$ and $k(z)\leq y$, 
    \item[(m4)] if $(k(x),y)\in R_2$, then there is   $z\in X_1$ such that $(x,z)\in R_1$ and $y\leq k(z)$,
    \item[(m5)]  if $(y,k(x))\in R_2$, then there is   $z\in X_1$ such that $(z,x)\in R_1$ and $ y\leq k(z)$ .
\end{itemize}
\end{definition}

\begin{lemma}\label{lt} Let $(X,R)$ be a tense H-space. If $x,y\in X$ and $(x,y)\notin R,$ then there exists $U\in D(X)$ such that $x\in f_{R}(U)$ and $y\in U$ or there exists $V\in D(X)$ such that $y\notin V$ and $x\notin g_{R}(V).$
\end{lemma}

\begin{proof} Let $x,y\in X$ such that $y\notin R(x)$. From (S2), we have that $y\notin \downarrow R(x)$ or $y\notin \uparrow R(x).$ In the first case, suppose that $y\notin \downarrow R(x)$. Then, for all $z\in R(x),$ $y\not\leq z$. As $X$ is Priestley space, for each $z\in R(x)$ there exists $U_z\in D(X)$ such that $y\in U_z$ and $z\notin U_z$. Then $R(x)\subseteq \bigcup\limits_{z\in R(x)}(X\setminus U_z)$. Since $R(x)$ is compact, from the preceding assertion we infer that there is $V\in D(X)$ such that $y\notin V$ and $R(x)\cap V=\emptyset$. It follows that $x\notin g_{R}(V).$ Similarlty, we can prove the second case.
\end{proof}

Next we will show that the category {\bf tHS} of tense H-spaces and tense H-functions is dually equivalent 
to the category {\bf tHA} of tense H-algebras and tense H-homomorphisms.

Now we will show a characterization of tense H-functions which will be
useful later.

\begin{lemma}\label{l5} Let $(X_1, R_1)$ and $(X_2, R_2)$ be two tense H-spaces. If $k:X_{1}\longrightarrow X_{2}$ is a tense H-morphism, then the following conditions are satisfied:

\begin{itemize}
        \item [{\rm (M1)}] $(x,y)\in R_1$ implies $(k(x),k(y))\in R_2$,
        \item [{\rm (M2)}] $k^{-1}(g_{R_2}(U))=g_{R_1}(k^{-1}(U))$ for any $U\in D(X_2),$
        \item [{\rm (M3)}] $k^{-1}(h_{R_2}(U))=h_{R_1}(k^{-1}(U))$ for any $U\in D(X_2),$
        \item [{\rm (M4)}] $k^{-1}(f_{R_2}(U))=f_{R_1}(k^{-1}(U))$ for any $U\in D(X_2),$
                \item [{\rm (M5)}] $k^{-1}(p_{R_2}(U))=p_{R_1}(k^{-1}(U))$ for any $U\in D(X_2).$
    \end{itemize}

\end{lemma}

\begin{proof} We will only prove (M2) and (M4).

\noindent (M2): Let $x\in k^{-1}(g_{R_2}(U))$. Hence, $k(x)\in g_{R_2}(U)$. By \ref{g}, we have that $R_{2}(k(x))\cap (X_2\setminus U)\not=\emptyset$. Then, there exists $y\in R_{2}(k(x))$ such that $y\notin U$. By (m2), there exists $z\in X_1$ such that $(x,z)\in R_{1}$ and $k(z)\leq y$. So, $k(z)\notin U$. Then, $z\in R_{1}(x)\cap (X_1\setminus k^{-1}(U))$. Therefore, $x\in g_{R_1}(k^{-1}(U))$. On the other hand, suppose that $x\in g_{R_1}(k^{-1}(U))$. By \ref{g}, there exists $y\in R_{1}(x)$ such that $k(y)\notin U$. Hence, by (m1) we have that $k(y)\in R_{2}(k(x))\cap (X_2\setminus U)$. So, $k(x)\in g_{R_2}(U)$. Therefore, $x\in k^{-1}(g_{R_2}(U))$.

\noindent (M4): Let $x\in k^{-1}(f_{R_2}(U))$ and suppose that $x\notin f_{R_1}(k^{-1}(U))$. Then, there exists $y\in R_{2}(k(x))$ such that $y\in U$. Hence, by (m4) there exists $z\in X_{1}$ such that $(x,z)\in R_{1}$ and $y\leq k(z)$. So, $k(z)\in U$. Then, from (M1), we have that $(k(x),k(z))\in R_{2}$. Therefore, $R_{2}(k(z))\cap U\not=\emptyset$ which is a contradiction. Now assume that $x\in f_{R_1}(k^{-1}(U))$ and suppose that $k(x)\notin f_{R_2}(U)$. Then, by \ref{f} there exists $y\in R_{2}(k(x))$ such that $y\in U$.  From this statement and (m4) there exists $z\in X_{1}$ such that $(x,z)\in R_{1}$ and $y\leq k(z)$. So, $k(z)\in U$. Therefore, $z\in R_{1}(x)\cap k^{-1}(U)$ which contradicts the hypothesis.
\end{proof}

\begin{lemma}\label{l6} Let $(X_1, R_1)$ and $(X_2, R_2)$ be two tense H-spaces. If $k:X_{1}\longrightarrow X_{2}$ is a Heyting morphism satisfying {\rm (M1)} to {\rm (M5)}, then $k:X_{1}\longrightarrow X_{2}$ is a tense H-morphism.  
\end{lemma}

\begin{proof} We will only prove (m2) and (m4).

\noindent (m2): Let $y\in R_{2}(k(x))$ and suppose that for any $z\in R_{1}(x)$ it is verified that $k(z)\not\leq y$. As $X$ is a Priestley space, for any $z\in R_{1}(x)$ there exists $U_z\in D(X_2)$ such that $k(z)\in U_z$ and $y\notin U_z$. Hence, $R_{1}(x)\subseteq \bigcup\limits_{z\in R_{1}(x)}k^{-1}(U_z)$. Since $k:X_{1}\longrightarrow X_{2}$ is a continuous function and $R_{1}(x)$ is compact, we infer that there is $U\in D(X_2)$ such that $y\notin U$ and $R_{1}(x)\cap (X_{1}\setminus k^{-1}(U))=\emptyset$,  from which it follows that $x\notin g_{R_1}(k^{-1}(U))$. Thus, by (M4) we have that $k(x)\notin g_{R_2}(U),$ which is a contradiction.

\noindent (m4): Let $y\in R_{2}(k(x))$ and suppose that for each $z\in R_{1}(x)$, $y\not\leq k(z)$. As $X$ is a Priestley space, for any $z\in R_{1}(x)$ there exists $U_z\in D(X_2)$ such that $y\in U_z$ and $k(z)\not\in U_z$. Hence, $R_{1}(x)\subseteq \bigcup\limits_{z\in R_{1}(x)}k^{-1}(X_2\setminus U_z)$. Since $k:X_{1}\longrightarrow X_{2}$ is a continuous function and $R_{1}(x)$ is compact, we infer that there is $U\in D(X_2)$ such that $y\in U$ and $R_{1}(x)\subseteq X_1\setminus k^{-1}(U)$. So, $x\in f_{R_1}(k^{-1}(U))$. Then, from (M4) we have $R_{2}(k(x))\subseteq X_2\setminus U$, which contradicts the hypothesis.
\end{proof}

As a consequence of the Lemma \ref{l5}  and Lemma \ref{l6} we obtain the following result:

\begin{proposition}\label{p1}
Let $(X_1, R_1)$ and $(X_2, R_2)$ be two tense H-spaces. Then, the following conditions are equivalent:

\begin{itemize}
    \item [{\rm (i)}] $k:X_{1}\longrightarrow X_{2}$ is a tense H-morphism,
    \item [{\rm (ii)}] $k:X_{1}\longrightarrow X_{2}$ is a Heyting morphism which satisfies {\rm (M1)} to {\rm (M5)}.
\end{itemize}

\end{proposition}

\begin{lemma}\label{lemma8} If $(X, R)$ is a tense H-space, then $(D(X),g_{R},h_{R},f_{R},p_{R})$ is a tense H-algebra.
\end{lemma}

\begin{proof}
We will only prove (T7) and (T8) from the definition of tense H-algebra. The rest of the properties can be checked without difficulty.

\begin{itemize}
    \item[(T7)] $g_{R}(U)\cap f_{R}(V)\subseteq g_{R}(U\cup V):$ Let $x\in g_{R}(U)\cap f_{R}(V)$. Hence, $R(x)\cap (X\setminus U)\not=\emptyset$ and $R(x)\subseteq (X\setminus V)$. Then, there exists $y\in R(x)\cap (X\setminus U)$ such that $y\notin V$. So, $R(x)\cap (X\setminus (U\cup V))\not=\emptyset$ i.e., $x\in g_{R}(U\cup V)$.
    \item[(T8)] $f_{R}(U\cap V)\subseteq f_{R}(U)\cup g_{R}(V):$ Let $x\in f_{R}(U\cap V)$ and suppose that $x\notin f_{R}(U)$. Hence, $R(x)\subseteq X\setminus (U\cap V)$ and $R(x)\not\subseteq (X\setminus U)$. Then, there exists $y\in R(x)$ such that $y\in U$. Since, $R(x)\subseteq (X\setminus U)\cup (X\setminus V)$, we have that $y\in X\setminus V$. So, $R(x)\cap (X\setminus V)\not=\emptyset$. Therefore, $x\in g_{R}(V).$
\end{itemize} 
\end{proof}

\begin{lemma}\label{lemma10} If $k : (X_1 , R_1)\longrightarrow (X_2, R_2)$ is a tense H-morphism, then the map $\Psi(k): D(X_2)\longrightarrow D(X_1)$ define by $\Psi(k)(U)=k^{-1}(U)$ is a homomorphism of tense H-algebras.
\end{lemma} 

\begin{proof} The Esakia duality for Heyting algebras allows us to state that $\Psi(k)$ is a Heyting homomorphism. Besides, taking into account that $k$ is a tense H-morphism, from Proposition \ref{p1}, we infer that $\Psi(k)$ commute with negative tense operators.
\end{proof}

\begin{definition}
Let $(\mathcal{A},N)$ be a tense H-algebra and let $R_{A}$ be the relation defined on $\mathfrak{X}(A)$ by:
\begin{equation}\label{relac}
(S,T)\in R_{A} \Longleftrightarrow f^{-1}(S)\subseteq X\setminus T\subseteq g^{-1}(S).    
\end{equation}
\end{definition}

\begin{remark}\label{remark1} Let $(\mathcal{A},N)$ be a tense H-algebra and let $S\in \mathfrak{X}(A)$. Then, it is easy to check that $R_{A}(S)=\downarrow R_{A}(S)\cap \uparrow R_{A}(S)$.
\end{remark}

\begin{lemma}\label{inversa} Let $(\mathcal{A},N)$ be a tense H-algebra and let $S,T\in \mathfrak{X}(A)$. Then, the following conditions are equivalent:
\begin{itemize}
    \item [{\rm (a)}] $f^{-1}(S)\subseteq X\setminus T\subseteq g^{-1}(S).$
    \item [{\rm (b)}] $p^{-1}(T)\subseteq X\setminus S\subseteq h^{-1}(T).$
\end{itemize}
\end{lemma}

\begin{proof}${\rm (a)}\Rightarrow {\rm (b)}:$  Let us suppose that $p(x)\in T$ and $x\in S$. From the preceding assertion and (T6), we have that $fp(x)\in S$. Then, $p(x)\notin T$, which is a contradiction. Thus, $p^{-1}(T)\subseteq X\setminus S$. On the other hand, let $z\in X\setminus S$ and suppose  that $h(z)\notin T$. Hence, $gh(z)\in S$. From the preceding assertion and (T5), we have that $z\in S$, which is a contradiction. Therefore, $X\setminus S\subseteq h^{-1}(T)$. The converse implication is similar.
\end{proof}

\begin{remark} Lemma \ref{inversa} means that we have two ways to define the relation $R_{A}$,
either by using $f$ and $g$, or by using $p$ and $h$.
\end{remark}

\begin{lemma}\label{lemma11}  Let $(\mathcal{A},N)$ be a tense H-algebra. Then, 
\begin{itemize}
    \item [{\rm (i)}] for all $S\in \mathfrak{X}(A),$ $R_{A}(S)$ is closed in $\mathfrak{X}(A),$
    \item [{\rm (ii)}]for all $S\in \mathfrak{X}(A),$ $R^{-1}_{A}(S)$ is closed in $\mathfrak{X}(A).$
\end{itemize}
\end{lemma}

\begin{proof} We will only prove {\rm (i)}. Similarly we can prove {\rm (ii)}. Suppose that $T\notin R_{A}(S)$. Then, by \ref{relac}, there exists $x\in f^{-1}(S)$ such that $x\in T,$ or there exists $y\in X\setminus T$ such that $y\notin g^{-1}(S)$. In the first case, $f(x)\in S$ and $T\in \sigma_{A}(x)$. Then taking into account that $f(x)\in S$ we infer that $\sigma_{A}(x)\cap R_{A}(S)=\emptyset$. From this assertion we deduce that $R_{A}(S)$ is closed in $\mathfrak{X}(A)$. In the second case, $T\notin \sigma_{A}(y)$ and $g(y)\notin S$. Since, $g(y)\notin S,$ we infer that $R_{A}(S)\subseteq \sigma_{A}(y)$. From this last assertion we can deduce that $R_{A}(S)$ is closed in $\mathfrak{X}(A)$ and
the proof is complete.
\end{proof}

The following lemma, can be proved using a similar technique that used in the proof of \cite[Lemma 2.14.]{ML2022}.

\begin{lemma}\label{lO} Let $(\mathcal{A},N)$ be a tense H-algebra and let $S,T\in \mathfrak{X}(A)$. Then, 

\begin{itemize}
    \item [{\rm (i)}] $X\setminus T\subseteq g^{-1}(S)\Longleftrightarrow (S,T)\in (R_{A}\circ \subseteq),$
    \item [{\rm (ii)}] $X\setminus S\subseteq h^{-1}(T)\Longleftrightarrow (S,T)\in (\supseteq\circ R_{A}),$
    \item [{\rm (iii)}] $f^{-1}(S)\subseteq X\setminus T\Longleftrightarrow (S,T)\in (\subseteq \circ R_{A}),$
    \item [{\rm (iv)}] $p^{-1}(T)\subseteq X\setminus S\Longleftrightarrow (S,T)\in (R_{A} \circ \supseteq).$
\end{itemize}
\end{lemma}

In what follows, if $(\mathcal{A},N)$ is a tense $H$-algebra, for $a\in A$  we will denote by $[a)$ ($(a]$) the filter  (ideal) generated by $\{a\}$.

\begin{lemma}\label{2.10} Let $(\mathcal{A},N)$ be a tense H-algebra and let $S\in \mathfrak{X}(A)$ and $a\in A$. Then,

\begin{itemize}
\item [{\rm (i)}] $g(a)\notin S$ iff  for all $T\in \mathfrak{X}(A)((S,T)\in R_{A}\Rightarrow a\in T),$
\item [{\rm (ii)}] $h(a)\notin S$ iff  if for all $T\in \mathfrak{X}(A)((T,S)\in R_{A}\Rightarrow a\in T),$
\item [{\rm (iii)}] $f(a)\in S$ iff  for all $T\in \mathfrak{X}(A)((S,T)\in R_{A}\Rightarrow a\notin T),$
\item [{\rm (iv)}] $p(a)\in S$ iff  for all $T\in \mathfrak{X}(A)((T,S)\in R_{A}\Rightarrow a\notin T).$
\end{itemize}
\end{lemma}

\begin{proof} We will only prove (i) and (iii).

\noindent (i)
$(\Rightarrow):$ Suppose that $g(a)\notin S$ and $(S,T)\in R_{A}$. Then, $f^{-1}(S)\subseteq X\setminus T\subseteq g^{-1}(S)$ and so $a\in T$.

\noindent $(\Leftarrow):$ Suppose that $g(a)\in S$. Then, $a\notin X\setminus g^{-1}(S)$. By Lemma \ref{lemma2}(b), $X\setminus g^{-1}(S)$ is a filter. Then, by (T9) we have that $(a]\cap (X\setminus g^{-1}(S))=\emptyset$. So, by the Birkhoff--Stone theorem, there exists $T\in \frak{X}(A)$ such that $X\setminus T\subseteq g^{-1}(S)$ and $a\notin T$. Hence, by Lemma  \ref{lO}, $(S,T)\in (R_{A}\circ \subseteq )$. Then, $(S,Q)\in R_{A}$ and $Q\subseteq T$ for some $Q\in \frak{X}(A)$. Therefore, $a\notin Q$.

\noindent (iii): $(\Rightarrow):$ Suppose that $f(a)\in S$ and $(S,T)\in R_{A}$. Then, $f^{-1}(S)\subseteq X\setminus T\subseteq g^{-1}(S)$ and so $a\notin T.$

\noindent $(\Leftarrow):$ Suppose that $f(a)\notin S$. So, $a\notin f^{-1}(S)$. By Lemma \ref{lemma2} (a), $f^{-1}(S)$ is an ideal. Then, by (T10) we obtain $[a)\cap f^{-1}(S)=\emptyset$. Hence, by the  Birkhoff--Stone theorem, there exists $Q\in \frak{X}(A)$ such that $f^{-1}(S)\subseteq X\setminus Q$ and $a\in Q$. By Lemma \ref{lO}, $(S,Q)\in (R_{A}\circ \supseteq)$. Then, $(S,T)\in R_{A}$ and $Q\subseteq T$ for some $T\in \frak{X}(A)$. Therefore,  $a\in T$. 
\end{proof}

\begin{lemma}\label{l13} Let $(\mathcal{A},N)$ be a tense H-algebra. Then, $$\Phi(A)=(\mathfrak{X}(A),\subseteq,R_{A})$$ is a tense H-space and $\sigma_{A}:A\longrightarrow D(\mathfrak{X}(A))$ is a tense H-isomorphism.
\end{lemma}

\begin{proof} From the duality for Heyting algebras, we have that $(\frak{X}(A),\subseteq)$ is a
Heyting space. First, we will prove that the following assertions hold for all $a\in A:$

\begin{equation}\label{gh}
g_{R_A}(\sigma_{A}(a))=\sigma_{A}(g(a));\,\,\, h_{R_A}(\sigma_{A}(a))=\sigma_{A}(h(a));     
\end{equation}
\begin{equation}\label{fp}
f_{R_A}(\sigma_{A}(a))=\sigma_{A}(f(a));\,\,\, p_{R_A}(\sigma_{A}(a))=\sigma_{A}(p(a)).    
\end{equation}

\noindent $\bullet$ $g_{R_A}(\sigma_{A}(a))=\sigma_{A}(g(a)):$  Suppose that $S\in g_{R_A}(\sigma(a))$. From this  there exists $T\in \frak{X}(A)$ such that $(S,T)\in R_{A}$ and $a\in T$. Since, $f^{-1}(S)\subseteq X\setminus T\subseteq g^{-1}(S),$ we have that $f(a)\in S,$ that is, $S\in \sigma_{A}(g(a))$ and therefore $g_{R_A}(\sigma_{A}(a))\subseteq \sigma_{A}(g(a)).$ On the other hand, suppose that $S\in \sigma_{A}(g(a))$, that is, $g(a)\in S$. Then, by Lemma \ref{2.10}, there exists $T\in \frak{X}{(A)}$ such that $(S,T)\in R_{A}$ and $a\notin T$. From which it follows that $T\in R_{A}(S)\cap (X\setminus \sigma_{A}(a))$. Therefore, $S\in g_{R_A}(\sigma_{A}(a)),$ from which we conclude $\sigma_{A}(g(a))\subseteq g_{R_A}(\sigma_{A}(a))$. 

\noindent $\bullet$ $f_{R_A}(\sigma_{A}(a))=\sigma_{A}(f(a)):$ Let us take a prime filter $S$ such that $f(a)\notin S$. By Lemma \ref{2.10}, there exists $T\in \frak{X}(A)$ such that $(S,T)\in R_{A}$ and $a\in T$. Then, $R_{A}^{-1}(S)\not\subseteq X\setminus \sigma_{A}(a)$. So, $S\notin f_{R}(\sigma_{A}(a))$ and, therefore,  $f_{R_A}(\sigma_{A}(a))\subseteq \sigma_{A}(f(a))$. Moreover, it is inmediate that $\sigma_{A}(f(a))\subseteq f_{R_A}(\sigma_{A}(a)).$

In a similar way we can prove that 

\begin{center}
$h_{R_A}(\sigma_{A}(a))=\sigma_{A}(h(a))$ and $p_{R_A}(\sigma_{A}(a))=\sigma_{A}(p(a))$. 
\end{center}

\noindent (S1): By Lemma \ref{lemma11}, $R_{A}(S)$ and $R_{A}^{-1}(S)$ are closed subsets of $\frak{X}(A)$ for all $S\in \frak{X}(A)$.

\noindent (S2): By Remarks \ref{remark1}, for any $S\in \frak{X}(A),$ $R_{A}(S)=\downarrow R_{A}(S)\cap \uparrow R_{A}(S).$

\noindent (S3): For any $U\in D(\frak{X}(A))$ there exists $a\in A$ such that $U=\sigma_{A}(a)$. Then, the equalities \ref{gh} and \ref{fp} allows us to affirm that \begin{center}$g_{R}(U), h_{R}(U), f_{R}(U), p_{R}(U)\in D(\frak{X}(A)).$\end{center}  So, $(\frak{X}(A),R_{A})$ is a tense $H$-space. By Lemma \ref{lemma8} we have that $D(\frak{X}(A))$ is a tense H-algebra. By virtue of the results established in \cite{Esakia} and the assertions \ref{gh} and \ref{fp} we conclude that $\sigma_{A}$ is a tense H-isomorphism.
\end{proof}

\begin{lemma}\label{l14} Let $(A_1,N1)$ and $(A_2,N_2)$ be two tense H-algebras and $k:A_{1}\longrightarrow A_{2}$. Then, the application $\Phi(k):\mathfrak{X}(A_2)\longrightarrow \mathfrak{X}(A_1),$ defined by $\Phi(k)(S)=k^{-1}(S)$ for all $S\in \mathfrak{X}(A_2),$ is a tense H-function.
\end{lemma}

\begin{proof} From the duality for Heyting algebras, it holds that the application $\Phi(k):\frak{X}(A_2)\longrightarrow \frak{X}(A_1)$ is a morphism of
Heyting algebras. In what follows, we will prove that (M1) and (M2) hold.

\noindent {\rm (M1)}: Let $S,T\in \frak{X}(A_2)$. Let us prove that if $(S,T)\in R_{A_2},$ then $f^{-1}(k^{-1}(S))$ $\subseteq X\setminus k^{-1}(T)\subseteq g^{-1}(k^{-1}(S))$. Suppose that $a\in f^{-1}(k^{-1}(S))$. Then we have that $k(f(a))=f(k(a))\in S,$ from which it follows that $k(a)\notin T,$ i.e., $a\notin k^{-1}(T)$. In a similar way we can prove $X\setminus k^{-1}(T)\subseteq g^{-1}(k^{-1}(S)).$

\noindent {\rm (M2)}: Let $U\in D(\frak{X}(A_1))$. Since, $\sigma_{A_1}:A_{1}\longrightarrow D(\frak{X}(A_1))$ is an isomorphism of tense H-algebras we have that $U=\sigma_{A_1}(a)$ for some $a\in A_{1}$. Besides, for any $a\in A_{1},$ we have that $\Phi(k)^{-1}(\sigma_{A_1}(a))=\sigma_{A_2}(k(a)).$ From this last statement and the fact that $k$ is a homomorphism of tense H-algebras we obtain that $g_{R_{A_2}}(\Phi(k)^{-1}(\sigma_{A}(a)))=g_{R_{A_2}}(\sigma_{A_2}(k(a)))=\sigma_{A_2}(g_2(k(a)))=\sigma_{A_2}(k(g_{1}(a)))=\Phi(k)^{-1}(\sigma_{A_1}(g(a)))=\Phi(k)^{-1}(g_{R_{A_1}}(\sigma_{A_1}(a))).$ Similarly, the axioms (M3), (M4) and (M5) can be proved.

\end{proof}

Lemmas \ref{l13} and \ref{l14} show that $\Phi$ is a contravariant functor from {\bf tHA} to {\bf tHS}.

The following characterization of isomorphisms in the category {\bf tHS} will be used to determine the duality we were looking for.

\begin{proposition} Let $(X_1,R_1)$ and $(X_2,R_2)$ be two tense H-spaces. Then, for every function $k:X_{1}\longrightarrow X_2$ the following conditions are equivalents:

\begin{itemize}
    \item [{\rm (i)}] $k$ is an isomorphism in the category {\bf tHS},
    \item [{\rm (ii)}] $k$ is a bijective H-morphism such that for all $x,y\in X_{1}:$
    \item [{\rm (m)}] $(x,y)\in R_{1}\Longleftrightarrow (k(x),k(y))\in R_{2}.$
\end{itemize}
\end{proposition}

\begin{proof} It is routine.
\end{proof}

The application $\varepsilon_{X}:X\longrightarrow \frak{X}(D(X))$ leads to another characterization of tense H-space, which also allows us to assert that this application is an isomorphism in the category {\bf tHS}, as we will describe below:

\begin{lemma} Let $(X,R)$ be a tense H-space and let $R_{D(X)}$ be the relation defined on $\frak{X}(D(X))$ by means of the operators $f_{R}$ and $g_{R}$ as follows:

\begin{equation}
(\varepsilon_{X}(x),\varepsilon_{X}(y))\in R_{D(X)}\Longleftrightarrow f_{R}^{-1}(\varepsilon_{X}(x))\subseteq \frak{X}(D(X))\setminus \varepsilon_{X}(y)\subseteq g_{R}^{-1}(\varepsilon_{X}(x)).
\end{equation}
Then, the following property holds:

\begin{itemize}
    \item [{\rm (S4)}] $(x,y)\in R$ implies $(\varepsilon_{X}(x),\varepsilon_{X}(y))\in R_{D(X)}.$
\end{itemize}
\end{lemma}

\begin{proof} Let us consider $x,y\in X$ such that $(x,y)\in R$ and prove that $$f_{R}^{-1}(\varepsilon_{X}(x))\subseteq \frak{X}(D(X))\setminus \varepsilon_{X}(y)\subseteq g_{R}^{-1}(\varepsilon_{X}(x)).$$ Let $U\in D(X)$ such that $f_{R}(U)\in \varepsilon_{X}(x)$. Then, $R(x)\subseteq X\setminus U$. This last assertion allows us to infer that $X\setminus U\in \varepsilon_{X}(y)$. Therefore, $f^{-1}(\varepsilon_{X}(x))\subseteq \frak{X}(D(X))\setminus \varepsilon_{X}(y)$. On the other hand, let us consider $V\in \frak{X}(D(X))\setminus \varepsilon_{X}(y)$. Then, $y\in R(x)\cap (\frak{X}(D(X))\setminus \varepsilon_{X}(y))$. So, $x\in g_{R}(V),$ i.e., $g_{R}(V)\in \varepsilon_{X}(x)$. Therefore, we have that $\frak{X}(D(X))\setminus \varepsilon_{X}(y)\subseteq g_{R}^{-1}(\varepsilon_{X}(x))$.
\end{proof}

\begin{proposition}\label{propo3} Let $(X , R)$ be a tense $H$-space. Then the map $\varepsilon_{X}:X\longrightarrow \frak{X}(D(X))$ is an isomorphism in the category {\bf tHS}.
\end{proposition}

\begin{proof}As a consequence of the Esakia duality for Heyting algebras, we have that $\varepsilon_{X}$ is
a bijective Heyting morphism. Furthermore, from (S4) and Lemma \ref{lt} it follows that $(x, y) \in R$
if and only if $(\varepsilon_X (x), \varepsilon_X (y)) \in R_{D(X)}$. 
\end{proof}

From the previous results and using the usual procedures we obtain the following
theorem.

\begin{theorem}
The categories $\mathbf{tHA}$ and $\mathbf{tHS}$ are dually equivalent.
\end{theorem}

\section*{Acknowledgements}
The authors want to thank the institutional support of Consejo Nacional de Investigaciones Cient\'ificas y T\'ecnicas (CONICET).



\AuthorAdressEmail{Federico Almi\~{n}ana}{CONICET and Instituto de Ciencias B\'asicas\\
Universidad Nacional de San Juan\\
5400, San Juan, Argentina}{federicogabriel17@gmail.com}

\AuthorAdressEmail{Gustavo Pelaitay}{CONICET and Instituto de Ciencias B\'asicas\\
Universidad Nacional de San Juan\\
5400, San Juan, Argentina}{gpelaitay@gmail.com}


\AdditionalAuthorAddressEmail{William Zuluaga}{ CONICET and Departamento de Matem\'atica\\ Facultad de Ciencias Exactas\\
Universidad Nacional del Centro  de la Provincia de Buenos Aires\\
Pinto 399, Tandil, Buenos Aires, Argentina}{wizubo@gmail.com}
\end{document}